\documentclass[11pt]{amsart}

\usepackage{amssymb,amsfonts,amsmath,graphicx,verbatim,eufrak,amsthm,calc}
\usepackage[foot]{amsaddr}
\usepackage{color}
\usepackage{xcolor}
\usepackage{newfloat}
\usepackage{hyperref}
\usepackage{tikz}
\hypersetup{
hidelinks,backref=true,pagebackref=true,hyperindex=true,
    colorlinks=true,       
    linkcolor=violet,          
    citecolor=green,        
    filecolor=magenta,      
    urlcolor=cyan           
}

\colorlet{genial}{black} 
\colorlet{genialsol}{black}

\makeatletter

\newtheoremstyle{genialnumbox}
{7pt}
{7pt}
{\normalfont}
{}
{\small\bf\sffamily\color{genial}}
{\;}
{0.25em}
{%
{\small\sffamily\color{genial}\thmname{#1}}%
{\nobreakspace\thmnumber{\@ifnotempty{#1}{}\@upn{#2}}}
\thmnote{{\nobreakspace\the\thm@notefont\sffamily\bfseries\color{black}\nobreakspace(#3)}} 
}

\newtheoremstyle{blacknumex}
{7pt}
{7pt}
{\normalfont}
{} 
{\small\bf\sffamily}
{\;}
{0.25em}
{%
{\small\sffamily\color{genial}\thmname{#1}}%
{\nobreakspace\thmnumber{\@ifnotempty{#1}{}\@upn{#2}}}
\thmnote{{\nobreakspace\the\thm@notefont\sffamily\bfseries\color{black}\nobreakspace(#3)}} 
}

\newtheoremstyle{blacknumbox} 
{7pt}
{7pt}
{\normalfont}
{}
{\small\bf\sffamily}
{\;}
{0.25em}
{%
{\small\sffamily\color{genial}\thmname{#1}}%
{\nobreakspace\thmnumber{\@ifnotempty{#1}{}\@upn{#2}}}
\thmnote{{\nobreakspace\the\thm@notefont\sffamily\bfseries\color{black}\nobreakspace(#3)}} 
}

\newtheoremstyle{genialnum}
{7pt}
{7pt}
{\normalfont}
{}
{\small\bf\sffamily\color{genial}}
{\;}
{0.25em}
{%
{\small\sffamily\color{genial}\thmname{#1}}%
{\nobreakspace\thmnumber{\@ifnotempty{#1}{}\@upn{#2}}}
\thmnote{{\nobreakspace\the\thm@notefont\sffamily\bfseries\color{black}\nobreakspace(#3)}} 
}

\makeatother

\RequirePackage[framemethod=default]{mdframed} 

\newmdenv[skipabove=7pt,
skipbelow=7pt,
rightline=false,
leftline=false,
topline=false,
bottomline=false,
backgroundcolor=black!5,
linecolor=genial,
innerleftmargin=5pt,
innerrightmargin=5pt,
innertopmargin=10pt,
leftmargin=0cm,
rightmargin=0cm,
innerbottommargin=10pt]{tBox}

\newmdenv[skipabove=7pt,
skipbelow=7pt,
rightline=false,
leftline=false,
topline=false,
bottomline=false,
backgroundcolor=genial!10,
linecolor=genial,
innerleftmargin=5pt,
innerrightmargin=5pt,
innertopmargin=5pt,
innerbottommargin=5pt,
leftmargin=0cm,
rightmargin=0cm,
linewidth=4pt]{eBox}	

\newmdenv[skipabove=7pt,
skipbelow=7pt,
rightline=false,
leftline=true,
topline=false,
bottomline=false,
linecolor=genial!50,
innerleftmargin=5pt,
innerrightmargin=5pt,
innertopmargin=5pt,
leftmargin=0cm,
rightmargin=0cm,
linewidth=4pt,
innerbottommargin=5pt]{dBox}	

\newmdenv[skipabove=7pt,
skipbelow=7pt,
rightline=false,
leftline=false,
topline=false,
bottomline=false,
linecolor=gray,
backgroundcolor=black!5,
innerleftmargin=5pt,
innerrightmargin=5pt,
innertopmargin=5pt,
leftmargin=0cm,
rightmargin=0cm,
linewidth=4pt,
innerbottommargin=5pt]{cBox}

\newmdenv[skipabove=7pt,
skipbelow=7pt,
rightline=false,
leftline=false,
topline=false,
bottomline=false,
linecolor=gray,
backgroundcolor=black!5,
innerleftmargin=5pt,
innerrightmargin=5pt,
innertopmargin=5pt,
leftmargin=0cm,
rightmargin=0cm,
linewidth=4pt,
innerbottommargin=5pt]{pBox}

\newmdenv[skipabove=7pt,
skipbelow=7pt,
rightline=false,
leftline=false,
topline=false,
bottomline=false,
linecolor=genialsol,
innerleftmargin=5pt,
innerrightmargin=5pt,
innertopmargin=0pt,
leftmargin=0cm,
rightmargin=0cm,
linewidth=4pt,
innerbottommargin=0pt]{solBox}	

\theoremstyle{genialnumbox}
\newtheorem{thm1}{Theorem}
\newtheorem{ithm1}[thm1]{$\star$ THEOREM}
\newtheorem{ques1}[thm1]{Question}
\newtheorem{conj1}[thm1]{Conjecture}

\theoremstyle{blacknumex}
\newtheorem{exer}[thm1]{Exercise}
\newtheorem{exer*}[thm1]{$\ast$ Exercise}

\theoremstyle{blacknumbox}
\newtheorem{dfn1}[thm1]{Definition}

\theoremstyle{genialnum}
\newtheorem{cor1}[thm1]{Corollary}
\newtheorem{prop1}[thm1]{Proposition}
\newtheorem{lem1}[thm1]{Lemma}

\newtheorem{exm1}[thm1]{Example}

\newenvironment{thm}{\paragraph{ } \begin{tBox}\begin{thm1}}{\end{thm1}\end{tBox}}

\newenvironment{exe*}{\paragraph{ } \begin{eBox}\begin{exer*}}{\hfill{\color{genial}
\ensuremath{\diamond\diamond\diamond}}\end{exer*}\end{eBox}}
\newenvironment{dfn}{\paragraph{ } \begin{dBox}\begin{dfn1}}{\end{dfn1}\end{dBox}}

\newenvironment{ques}{\paragraph{ } \begin{cBox}\begin{ques1}}{\end{ques1}\end{cBox}}	
\newenvironment{conj}{\paragraph{ } \begin{cBox}\begin{conj1}}{\end{conj1}\end{cBox}}	

\newenvironment{prop}{\paragraph{ } \begin{pBox}\begin{prop1}}{\end{prop1}\end{pBox}}	
\newenvironment{lem}{\paragraph{ } \begin{pBox}\begin{lem1}}{\end{lem1}\end{pBox}}

\newenvironment{lem*}[1]{\vspace{1ex}\noindent
{\bf Lemma* (#1).} [restatement]  \hspace{0.5em} \em }{ }

\newenvironment{thm*}[1]{\begin{cBox}
\vspace{1ex}\noindent 
{\bf Theorem* (#1).} [restatement]  \hspace{0.5em} }{\end{cBox}}

\theoremstyle{genialnum}

\newtheorem*{clm*}{Claim}

\newenvironment{sol}%
{\begin{solBox}
\par \noindent 
\scriptsize
{\bf Solution to ex:{\color{blue} \arabic{thm1}}.}  {\color{red} \ \  :( } \\ }%
{\hfill {\color{blue} :) $\checkmark$} \end{solBox}}

\newcommand{\ENDEXER}{
{\expandafter\comment}
{\expandafter\endcomment}
}

\newtheorem{rem}[thm1]{Remark}

\makeatletter
\renewcommand{\@seccntformat}[1]{\llap{\textcolor{genial}{\csname the#1\endcsname}\hspace{1em}}}                    
\renewcommand{\section}{\@startsection{section}{1}{\z@}
{-4ex \@plus -1ex \@minus -.4ex}
{1ex \@plus.2ex }
{\normalfont\large\sffamily\bfseries}}
\renewcommand{\subsection}{\@startsection {subsection}{2}{\z@}
{-3ex \@plus -0.1ex \@minus -.4ex}
{0.5ex \@plus.2ex }
{\normalfont\sffamily\bfseries}}
\renewcommand{\subsubsection}{\@startsection {subsubsection}{3}{\z@}
{-2ex \@plus -0.1ex \@minus -.2ex}
{.2ex \@plus.2ex }
{\normalfont\small\sffamily\bfseries}}                        
\renewcommand\paragraph{\@startsection{paragraph}{4}{\z@}
{-2ex \@plus-.2ex \@minus .2ex}
{.1ex}
{\normalfont\small\sffamily\bfseries}}

\makeatother

\newcommand{\IP}[1]{\left\langle #1 \right\rangle}

\newcommand{\set}[1]{\left\{#1\right\}}

\newcommand{\N}{\mathbb{N}}

\newcommand{\R}{\mathbb{R}}

\newcommand{\eps}{\varepsilon}

\newcommand{\ie}{{\em i.e.\ }}
\newcommand{\eg}{{\em e.g.\ }}

\DeclareMathOperator{\E}{\mathbb{E}}     

\renewcommand{\Pr}{}
\let\Pr\relax
\DeclareMathOperator{\Pr}{\mathbb{P}}

\newcommand{\1}[1]{\mathbf{1}_{\set{ #1 } }}

\newcommand{\ind}[1]{\mathbf{1}_{ #1}}

\def\squareforqed{\hbox{\rlap{$\sqcap$}$\sqcup$}}
\def\qed{\ifmmode\squareforqed\else{\unskip\nobreak\hfil
\penalty50\hskip1em\null\nobreak\hfil\squareforqed
\parfillskip=0pt\finalhyphendemerits=0\endgraf}\fi}

\newcommand{\ignore}[1]{ }

\newcommand{\define}[1]{\textbf{#1}}

\newcommand{\HF}{\mathsf{HF}}

\renewcommand{\S}{\mathcal{S}}

\newcommand{\haar}{\mathfrak{m}}

\title{Polynomially growing harmonic functions on connected groups}

\author[Idan Perl]{Idan Perl}
\address{Ben-Gurion University of the Negev, Be'er Sheva ISRAEL}
\email{IP: perli@bgu.ac.il}

\author[Ariel Yadin]{Ariel Yadin}
\email{AY: yadina@bgu.ac.il}

\thanks{Supported by the Israel Science Foundation (grant no.\ 1346/15). \\ 
We thank Yair Glasner and Tom Meyerovitch for helpful discussions.
}

\begin{document}
\maketitle

\begin{abstract}
We study the connection between the dimension of certain spaces of harmonic functions 
on a group and its geometric and algebraic properties.

Our main result shows that (for sufficiently ``nice'' random walk measures) a 
connected, compactly generated, locally compact group has polynomial 
volume growth if and only if the space of linear growth
harmonic functions has finite dimension.

This characterization is interesting in light of the fact that Gromov's theorem 
regarding finitely generated groups of polynomial growth does not 
have an analog in the connected case. That is, there are examples of connected groups of polynomial 
growth that are not nilpotent by compact.
Also, the analogous result for the discrete case has only been established for solvable groups,
and is still open for general finitely generated groups.
\end{abstract}

\section{Introduction}

\subsection{Background}

The study of harmonic functions on abstract groups has been quite fruitful in the past few decades.
Bounded harmonic functions have a deep algebraic structure and have been used to study 
``boundaries'' of groups, especially (but not only) in the discrete case.
This topic was initiated by Furstenberg \cite{F63, F73}.
A search for ``Poisson-Furstenberg boundary'' will reveal an immense amount of literature, 
we refer to \cite{KV83, furman2002RW} and references therein for the interested reader.
As for unbounded harmonic functions, positive harmonic functions were studied in the Abelian case 
by Chouqet \& Deny \cite{CD60} (and further by Raugi  \cite{Raugi} for nilpotent groups).
Yau \cite{Yau} studied  positive harmonic functions on open manifolds of non-negative Ricci curvature.
He conjectured that the space of harmonic functions that grow at most like some polynomial 
on such a manifold should have finite dimension.
This was proved by Colding and Minicozzi \cite{ColMin}.
Kleiner \cite{Kleiner} 
used Colding and Minicozzi's approach for finitely generated groups of polynomial growth,
to reprove Gromov's theorem regarding such groups \cite{Gromov}.

These works bring to light a connection between algebraic properties (nilpotence),
analytic properties (harmonic functions and random walks) 
and geometric properties (volume growth, curvature).
They motivate the following meta-questions:
Given a group $G$, and some space of harmonic functions on $G$,
what can be said about the dimension of the space and its relation to the algebraic and 
geometric properties of the group?  Is the 
dimension independent of the choice of specific random walk?
Does the finiteness of the dimension depend only on the group's algebraic properties? 
In general, one would like to understand
the structure of representations of the group given by its 
canonical action on some specific space of harmonic functions; 
how do these representations vary
as the underlying random walk measure is changed?

An example for a precise formulation of one such question is the following conjecture, 
which has been open for quite some time.

\begin{conj}
Let $G$ be a compactly generated locally compact group.
Let $\mu, \nu$ be two symmetric, adapted probability measures on $G$,
with an exponential tail.  Then, $(G,\mu)$ is Liouville if and only if $(G,\nu)$ is Liouville. 
\end{conj}

Here $(G,\mu)$ is Liouville means that any bounded $\mu$-harmonic function is constant.
It is well known 
that the space of bounded harmonic functions is either only the constant functions (\ie Liouville)
or has infinite dimension.
(For finitely generated groups this is 
also an easy consequence of Theorem \ref{thm: convergence is infinite dim} 
or Theorem \ref{thm: convRW} below.)
So an equivalent formulation of the above conjecture is that the dimension of the space of bounded harmonic functions does not depend on the specific choice of (nicely behaved) measure $\mu$.

As stated, this question regarding bounded harmonic functions has been open for a while.
This is part of the motivation for the following conjecture, 
from  \cite{MY16}.

\begin{conj} \label{conj:main}
Let $G$ be a compactly generated locally compact group.
Let $\mu$ be a symmetric, adapted probability measure on $G$,
with an exponential tail.  Then, $G$ has polynomial growth if and only if
the space of linearly growing $\mu$-harmonic functions on $G$ is finite dimensional.
\end{conj}

Note that a group $G$ with measure $\mu$ may be Liouville but still have an infinite dimension
of linearly growing harmonic functions (see e.g.\ \cite{KV83, MY16} and below for examples).

In \cite{MY16} this conjecture is verified for $G$ finitely generated and (virtually) solvable.
In fact, it is known that for finitely generated $G$, 
the dimension of the space of linear growth harmonic functions 
is either infinite or some number independent of the choice of specific measure,
see \cite{MPTY}.

The main result of this paper is a proof of Conjecture \ref{conj:main} for {\em connected}
topological groups.  In order to precisely state the results we introduce some notation.

\subsection{Notation and main results}

Let $G$ be a compactly generated locally compact (CGLC) group, 
and fix $K$ a compact generating set. 
Assume it is symmetric (\ie $K=K^{-1} = \{ x^{-1} :  x \in K \}$). 
Let $K^n = \{ x_1 x_2 \cdots x_n \ : \ x_1,\ldots,x_n \in K \}$.
$K$ induces a metric on $G$ by
$$
d_K(x,y)=d_K(1,x^{-1}y):=\min \{n:\ x^{-1}y\in K^n\}
$$ 
and we use the notation $|x| = |x|_K = d_K(1,x)$. Note that this metric is left invariant, 
that is $d_K(x,y) = d_K(zx,zy)$,
and that for two choices of generating sets $K_1$ and $K_2$, 
the respective metrics are bi-Lipschitz, \ie there exists a constant $c = c(K_1,K_2)>0$ such that 
$
c^{-1} |x|_{K_1} \leq |x|_{K_2}\leq c|x|_{K_1}
$ 
for all $x\in G$. 

\subsubsection{Growth of a group}
The \textit{growth of} $G$ is the growth rate of the sequence $(\haar(K^n) )_n$, 
where $\haar = \haar_K$ is the Haar measure on $G$ normalized to $\haar (K) = 1$.
We are mainly interested in polynomial growth: $G$ is said to have \define{polynomial growth}
if there exist constants $C>0 , k>0$ such that for all $n \geq 1$ we have $\haar(K^n) \leq C n^k$.
$G$ is said to have \define{exponential growth} if there exists some $t>1$ and $c>0$ 
such that $\haar(K^n) \geq c t^n$ for all $n \geq 1$.  (When $G$ is a {\em connected} CGLC group,
the growth is always either polynomial of exponential, see \cite{jenkins}.)
Because of the bi-Lipschitz property of the different possible metrics, the growth of $G$ does not 
depend on the specific choice of $K$.

\subsubsection{Growth of functions}
For functions $f:G\to\R$, define the following (perhaps infinite) quantity: 
\begin{align*}
||f||_k=\limsup_{r\to\infty} r^{-k} \sup_{|x|\leq r	}|f(x)|.
\end{align*}
We say that $f:G\to\R$ has \textit{degree-$k$ polynomial growth} if $||f||_k<\infty$. In the case $k=1$ we say that $f$ has linear growth. Note that $||f||_k<\infty$ is equivalent to the existence of a constant $c>0$ such that $|f(x)|\leq c (1+|x|)^k$ for all $x\in G$. The group $G$ acts naturally on $\R^G$ by $(\gamma.f)(x)=f(\gamma^{-1}x)$. 
By bi-Lipschitzness, the property $|| f ||_k < \infty$ is independent of the choice of specific generating set 
(although the specific value of $|| f ||_k$ does depend on the metric induced by $K$). 

The reader should beware to not confuse the growth of the group, and the growth of a function on the group,
which are two different notions.

\subsubsection{Laplacian and harmonic functions}
Throughout, we will consider a probability measure $\mu$ on $G$. We will always assume that
\begin{itemize}
\item It is {\em adapted}, i.e.\ there is no proper closed subgroup $H\lneq G$ such that $\mu(H)=1$.
\item It is {\em symmetric}, i.e.\ $\mu(A)=\mu(A^{-1})$ for any measurable set $A$.
\item It has a third moment, i.e.\ $\int_G |s|^3 d\mu(s)<\infty$.
\end{itemize}
For short, we call a probability satisfying these three assumption \textit{courteous}. 
If $\mu$ satisfies $\int_G e^{\eps |s|} d \mu(s) < \infty$ for some $\eps>0$, 
we say that {\em $\mu$ has an exponential tail}.
Note that the property of having a third moment or of having an exponential tail is 
independent of specific choice of generating set, again because the different metrics are bi-Lipschitz.

For measurable functions $f:G\to\R$, we define the Laplace operator by
\begin{align*}
(\Delta_{\mu}f)(x)=f(x)-\int_G f(xs)d\mu(s),
\end{align*} 
and we say that a function $f:G\to\R$ is $\mu$-\textit{harmonic} if $\Delta_{\mu}f \equiv 0$.

We can now define the space of $\mu$-harmonic functions with polynomial growth of degree at most $k$:
\begin{align*}
\HF_k(G,\mu):=\{f:G\to\R \ | \ \Delta_{\mu}f \equiv 0,\ ||f||_k<\infty \ , \ f \textrm{ is continous } \}.
\end{align*}
Note that since $G$ acts on the left and harmonicity is checked on the right, $\HF_k(G,\mu)$ is a $G$-invariant subspace of $\R^G$.

\subsection{Main result: characterization of polynomial growth}

\label{scn: characterization}

As mentioned, Gromov's theorem \cite{Gromov} characterizes the geometric property of polynomial
growth of a finitely generated group by the 
algebraic property of containing a finite index nilpotent subgroup.
However, in the connected case, there is no such characterization.
In fact, it is not true that any CGLC group of polynomial growth is nilpotent by compact.
One can construct a connected $2$-step solvable linear group of polynomial growth that is not nilpotent by compact, see \cite[Example 7.9]{Bre07}.
Is is known that connected CGLC groups have either polynomial or exponential growths.
In fact, Jenkins \cite{jenkins} proves some equivalent conditions to polynomial growth,
one of which is not containing a free uniformly-discrete semigroup.  See \cite{jenkins} for details.

Our main result is the following theorem characterizing connected CGLC groups of polynomial growth
using an analytic property, namely the finiteness of the dimension of $\HF_1$.

\begin{thm} \label{thm:characterization}
Let $G$ be a connected CGLC group.  
Let $\mu$ be a courteous measure with exponential tail.
The following are equivalent:
\begin{enumerate}
\item $G$ has polynomial growth.
\item For any $k \geq 1$
we have $\dim \HF_k(G,\mu) < \infty$.
\item $\dim \HF_1(G,\mu) < \infty$.
\item The space $\HF_1(G,\mu)$ does not contain a non-constant positive function.
\end{enumerate}
\end{thm}

This is a solution of Conjecture \ref{conj:main} for the connected case.

Here is a sketch of the main steps of the argument.

The {\bf first step} uses the result of \cite{BBE97} stating  
that if $G$ is a closed subgroup of the $d$-dimensional 
affine group, $\S_d$, and if $G$ does not have polynomial growth, then there exists a non-constant positive continuous harmonic function $h$ on $G$. 

The {\bf second step} is to show that the above positive 
function has linear growth. 
This was shown in \cite{BB15} for the case $d=1$, and we extend their result to general $d \geq 1$. 

These two steps culminate in:

\begin{lem}
\label{lem: linear growth}
Let $G$ be a closed subgroup of $\S_d$ 
and let $\mu$ be a courteous measure on $G$. 
Suppose $G$ does not have polynomial growth, 
Then, there exists a continuous, non-constant positive 
$\mu$-harmonic function on $G$, which admits linear growth. 
\ie $\HF_1(G,\mu)$ contains a non-constant positive function.
\end{lem}

The proof of this lemma is in Section \ref{scn:linear growth HFs}.

The {\bf third step} is a reduction from general connected CGLC groups to the case of $\S_d$.
This step utilizes heavy machinery such as the solution to Hilbert's fifth problem,
which enables understanding of the structure of connected CGLC groups.
Ultimately, this third step proves the following lemma.
The proof is given in Section \ref{scn:from LC to Sd}.

\begin{lem}
\label{lem: exists HF}
Let $G$ be a connected CGLC group and let $\mu$ be a courteous 
measure on $G$.  
Then, there exists a connected closed subgroup $G'$ 
of the $d$-dimensional affine group $\S_d$, and a courteous measure $\mu'$ on $G'$ such that:
\begin{itemize}
\item $\dim \HF_1(G',\mu') \leq \dim \HF_1(G,\mu)$.
\item If $G'$ has polynomial growth then also $G$ has polynomial growth.
\item If there exists a non-constant positive function in $\HF_1(G',\mu')$ then there also
exists a non-constant positive function in $\HF_1(G,\mu)$.
\end{itemize}
\end{lem}

For the {\bf fourth and last step}, we want to show that $\HF_1$ has infinite dimension
as soon as it admits some non-constant positive function. 
This is the content of the following theorem, which may be of independent interest.

\begin{thm}
\label{thm: convergence is infinite dim}
Let $G$ be an amenable CGLC group. Let $\mu$ be a courteous measure on $G$. 

If $\dim \HF_1(G,\mu) < \infty$, 
then any $h \in \HF_1(G,\mu)$ which is positive must be constant.
\end{thm}

\begin{proof}[Proof of Theorem \ref{thm:characterization}]
In \cite{Perl18} it is shown that for any CGLC group of polynomial growth, $G$, and any courteous $\mu$ with exponential tail on $G$, the dimension of $\HF_k(G,\mu)$ is finite for all $k \geq 1$. (This is an extension of Kleiner's work \cite{Kleiner} to non-compactly-supported measures,
and to connected CGLC groups.) This gives the implication $(1) \Rightarrow (2)$.

$(2) \Rightarrow (3)$ is trivial.

$(3) \Rightarrow (4)$ follows from Theorem \ref{thm: convergence is infinite dim}.

For $(4) \Rightarrow (1)$: 
Assuming (4), 
by Lemma \ref{lem: exists HF}, there exists $G' \leq \S_d$ a closed connected subgroup 
of the affine group $\S_d$, and a courteous measure $\mu'$ on $G'$
such that $\HF_1(G',\mu')$ does not contain a non-constant positive function.
By  Lemma \ref{lem: linear growth}, it follows that $G'$ must have polynomial growth.
By Lemma \ref{lem: exists HF} again, $G$ has polynomial growth as well.
\end{proof}

\begin{rem}
The exponential tail property of $\mu$ was only used to prove $(1) \Rightarrow (2)$.
It is basically there because the measure $\mu$ needs to have good enough decay 
for the Laplacian to be well defined on polynomially growing functions.

Our proof actually shows that $(3) \Rightarrow (4) \Rightarrow (1)$ even when $\mu$ is only assumed
to be courteous, without the exponential tail assumption.
\end{rem}

\subsection{Convergence along random walks}

Let us stress here that Theorem \ref{thm: convergence is infinite dim} holds in the non-connected case as well, that is, for finitely generated groups. This may be of independent interest in other contexts. 
The main idea behind the proof of Theorem \ref{thm: convergence is infinite dim}
is Theorem \ref{thm: convRW} which states that if a function
converges a.s.\ along the random walk and has sub-exponential growth, 
then it must be constant. 
This is relevant to positive harmonic functions since a positive harmonic function evaluated
on the corresponding random walk provides a positive martingale, which converges a.s.\
by the martingale convergence theorem (see \cite{Durrett}).

\begin{dfn}
A function $f:G \to \R$ {\bf converges along random walks}
if the sequence $(f(xX_t))_t$ converges a.s.\ for any $x \in G$, 
(where $(X_t)_t$ is the $\mu$-random walk).
\end{dfn}

For example, as mentioned above, 
any positive harmonic function converges along random walks.
As do bounded harmonic functions.
Hence, Theorem \ref{thm: convergence is infinite dim} follows from the following theorem,
which may be of independent interest.

\begin{thm}
\label{thm: convRW}
Let $G$ be an amenable CGLC group and let $\mu$ be a courteous measure on $G$.
Let $(X_t)_t$ denote the $\mu$-random walk.
Let $f:G \to \R$ be a continuous function such that $f$ converges along random walks.
Assume that $f$ has sub-exponential growth; that is,
$$ \limsup_{r \to \infty} \tfrac1r \sup_{|x| \leq r } \log |f(x)|  = 0 . $$

If $\dim \mathrm{span} (G.f) < \infty$ then $f$ is constant.
\end{thm}

The proof is carried out in Section \ref{sec: infinite dim'l orbit}. 

The assumption of sub-exponential growth in 
Theorem \ref{thm: convRW} is technical, and most probably superfluous.
In Section \ref{sec: infinite dim'l orbit} we will show that in the discrete case,
where $G$ is finitely generated, this assumption is not actually required.
We conjecture that the assumption of sub-exponential growth can be removed 
in the general CGLC case, see Conjecture \ref{conj: no sub exp} below.
In addition, although the proof heavily uses the amenability of $G$, it is not clear that
this is a necessary condition for Theorem \ref{thm: convRW} to hold. 
Again, see the open questions below.

\subsection{Further questions}

These open questions are motivated by the results mentioned above.

In \cite{MPTY} it is shown that for a finitely generated group $G$ and courteous measure $\mu$ with exponential tail, if $\dim \HF_k(G,\mu) < \infty$ then the space $\HF_k(G,\mu)$ is basically the space of {\em harmonic polynomials} on $G$
of degree at most $k$ (see \cite{MPTY} for precise definitions). This proves that $\dim \HF_k(G,\mu) \in \{ \infty , d \}$ for some $d$ which depends only on the group $G$ and not on $\mu$.

\begin{conj}
Let $G$ be a CGLC group. Let $\mu ,\nu$ be courteous measures on $G$ with an exponential tail. Then $\dim \HF_k(G,\mu) = \dim \HF_k(G,\nu)$ for any $k \geq 0$.
\end{conj}

Note the we have also included the $k=0$ case in the above conjecture, i.e.\ the space of bounded harmonic functions.

We have seen that the finiteness of the dimension of $\HF_1$ characterizes polynomial growth (at least for connected groups and for solvable groups).
In the connected case, the same solvable linear non-nilpotent-by-compact example mentioned above from \cite{Bre07} shows that one can no longer obtain results analogous to \cite{MPTY}, since this group has finite dimensional $\HF_1$ but linear growth harmonic functions which are not polynomials. If the group $G$ is nilpotent however, one can show that functions in $\HF_k$ are polynomials even in the connected case.

\begin{ques}
Let $G$ be a connected CGLC group. Let $\mu$ be a courteous measure on $G$ with an exponential tail. Let $P^k(G)$ denote the space of polynomials of degree at most $k$ on $G$ (see \cite{MPTY}). Fix $k \geq 1$. Is it true that $\HF_k(G,\mu) \subset P^k(G)$ if and only if $G$ is nilpotent? 
\end{ques}

The results of Choquet \& Deny \cite{CD60}, Raugi \cite{Raugi}, and Yau \cite{Yau} motivate the question of the existence of non-constant positive harmonic functions 
on some group. Specifically:

\begin{ques}
Let $G$ be a CGLC group of non-polynomial growth.
Let $\mu$ be a courteous measure with an exponential tail on $G$.

Is it true that there exists a positive $\mu$-harmonic function that is non-constant?

Is it true that there exists such a function of linear growth?
\end{ques}

Let us remark that our proof of Theorem \ref{thm: main} answers the above question affirmatively, in the connected case (see the characterization in Section \ref{scn: characterization}). 
The results of \cite{MY16} also provide an affirmative answer in the finitely generated solvable case.

It is known that any finitely generated group of exponential growth admits a non-constant positive 
harmonic function, as observed in \cite{AK17}, although the function constructed there may have 
exponential growth, so does not necessarily belong to $\HF_k$.
For finitely generated groups in general we do not know the answer, even for some specific examples, \eg the Grigorchuk groups.

Regarding Theorem \ref{thm: convRW}, as mentioned above, the condition of sub-exponential growth 
seems to be superfluous.

\begin{conj}
\label{conj: no sub exp}
Let $G$ be an amenable CGLC group and $\mu$ a courteous measure on $G$.
Let $f:G \to \R$ be a continuous function that converges along random walks.

If $\dim \mathrm{span} (G.f)  < \infty$ then $f$ is constant.
\end{conj}

Finally, we only know how to prove Theorem \ref{thm: convRW} for amenable groups.
It would be quite surprising if this theorem does not hold in the non-amenable case.
Precisely:
\begin{ques}
If $G$ is a CGLC group, 
is it true that for any continuous function $f:G \to \R$ that converges along random walks,
if $f$ is non-constant its orbit spans an infinite dimensional space?
\end{ques}

\section{Linear growth positive harmonic function}
\label{sec: overview}

\subsection{Stationary measure on $\S_d$}

Denote by $\S_d$ the group of affine similarities on $\R^d$. An element of $\S_d$ is $g=(a,k,b)$ where $a\in (0,\infty) , k\in O(d),b\in \R^d$. Here $O(d)$ is the group of $d \times d$ orthogonal real matrices. (Note that in the $1$-dimensional case, $\S_1$, we may omit the $k$-coordinate).
The group's multiplication is defined by
\begin{align*}
g_1\cdot g_2=(a_1,k_1,b_1)\cdot (a_2,k_2,b_2):=(a_1a_2,k_1k_2, a_1k_1 b_2+b_1).
\end{align*}
$\S_d$ acts from the left on $\R^d$ by $(a,k,b).x=ak x+b$. For an element $g=(a,k,b)$, let $a(g)=a$.

Let $G$ be a closed subgroup of $\S_d$, and let $\mu$ be a courteous probability measure on $G$. 
Let $\nu$ be a measure on $\R^d$, and define
\begin{align*}
\mu \ast \nu(A)=\int_G \int_{\R^d}1_A(g.x)d\nu(x)d\mu(g).
\end{align*}
for any measurable set $A$. 

A Radon measure $\nu$ on $\R^d$ is called
$\mu$-\textit{stationary} if $\mu \ast \nu=\nu$. 

\begin{rem}
In \cite{BBE97} and related texts, a measure satisfying $\mu \ast \nu = \nu$ is called {\em $\mu$-invariant}.
One should be careful to distinguish between invariance of $\nu$ with respect to convolution with
the measure $\mu$, and the different notion of 
a {\em $G$-invariant measure}, which means $g.\nu = \nu$ for all $g \in G$.  
In this text we only deal with the former. In order to avoid confusion, 
we prefer the terminology {\em stationary}.
\end{rem}

The existence of a $\mu$-stationary Radon measure is shown in \cite{BBE97}. 
Namely, we have:

\begin{lem}[Proposition 1.1 in \cite{BBE97}]
\label{lem: invariant msr}
Let $G$ be a closed CGLC subgroup of $\S_d$, and $\mu$ a courteous measure on $G$. 
Then there exist a $\mu$-stationary (unbounded) Radon measure $\nu$ on $\R^d$. 
\end{lem}

\subsection{Linear growth harmonic functions}
\label{scn:linear growth HFs}

Let $\nu$ be a $\mu$-stationary Radon measure on $\R^d$. 
Let $\phi:\R^d\to\R$ be a compactly supported function. 
Define
\begin{align}
\label{dfn: harmonic function}
h(g):=\int_{\R^d} \phi(g.x)d\nu(x).
\end{align}
It is straightforward to check that $h$ is a $\mu$-harmonic function on $G$,
because $\nu$ is $\mu$-stationary. 
If $\phi \geq 0$ then $h \geq 0$. If $\ind{A_1}\leq \phi \leq \ind{A_2}$ 
for some measurable sets $A_1\subset A_2$, 
then $\nu(g.A_1) \leq h(g) \leq \nu(g.A_2)$ for all $g \in G$. 
Also, if $\phi$ is continuous, then $h$ is continuous as well.

Let $B=[-1,1]^d\subset \R^d$. Fix a compactly supported continuous function 
$\phi : \R^d \to \R$ such that $\ind{ (1/2) B} \leq \phi \leq \ind{B}$. 
Define $h$ as in \eqref{dfn: harmonic function}. By 
Lemma 2.10 in \cite{BE95} $h$ is non-constant as soon as $G$ does not have polynomial growth
(since $\nu$ can be chosen so that it is not $G$-invariant).

We want to show that $h$ has linear growth, \ie that there exists a constant $c_h>0$ such that 
$h(g)\leq c_h(1+|g|)$ for all $g\in G$. 
By compactness and the continuity of the action, there exists a constant $M>1$ such that 
$k.B\subset [-M,M]^d$ for all $k$ in the compact symmetric generating set $K$. 
By induction, this implies $g.B\subset [-M^{|g|},M^{|g|}]^d$ for all $g\in G$. 
Hence, to show linear growth of $h$, it 
will be suffice to show that 
\begin{align}
\label{measure growth}
\nu([-z,z]^d)\leq C(1+\log z) \qquad \forall  \ z>1
\end{align}
for some constant $C>0$. In \cite[Proposition 3.1(1)]{BB15}, this is shown for the case of $d=1$. 
Their proof relies on the total ordering of the real numbers, 
hence it does not generalize to $\R^d$ in a straightforward manner. 
We now use results from \cite{BB15} to prove the general $d$-dimensional case.

Let $\psi=(a,k,b)$ be a random element generated by the probability measure $\mu$ on $\S_d$. 
We assume the following:
\begin{itemize}
\item Recurrence: $\E[\log(a)]=0$ and $\Pr [ a = 1 ] \neq 1$. 
\item Non-degeneracy: $\Pr[\psi.x=x]<1$ for all $x\in \R^d$.
\item Moment condition: $\E [\big(\ |\log(a)|+\log(1+||b||)\ \big)^3]<\infty$. 
\end{itemize}
We note that by \cite{Elie82,Gui80}, there exist constants $C,D>0$ such that  
\begin{align*}
C^{-1} |\psi|-D< |\log(a)|+\log(1+||b||) < C|\psi|+D.
\end{align*}
Hence, the above moment condition is equivalent to existence of a third moment of $\mu$.

In the above, and throughout this section, the $|| \cdot ||$ norm on $\R^d$ is the $L^\infty$-norm;
\ie $|| x || = \max_{1 \leq j \leq d} |x_j|$.

Next, for an element $\psi$ as above, we define 
\begin{align}
\label{g}
g_{\psi}:= \big( a, \max\{||b|| , 1\} \ \big) \in \S_1.
\end{align} 
We denote by $(\psi_t)_{t\geq 0}$ a sequence of i.i.d.\ $\mu$ random elements,
and abbreviate by $(g_t)_{t\geq 0}$ the induced sequence $(g_{\psi_t})_{t\geq 0}$. 
For an element $g=(a,b)\in\S_1$, denote $a(g)=a, b(g)=b$. 
Finally, we define $R_0 = 0$ and $\Psi_0 = (1,I,0)$ as the identity element in $\S_d$,
as well as 
$$ R_t =g_1\cdots g_t \qquad \qquad \Phi_t = \psi_1 \cdots \psi_t . $$

Note that 
\begin{align}
\label{b}
b(R_{t+1}) &=a(R_t)b(g_{t+1})+b(R_t) , 
\end{align}
which implies that also
\begin{align}
\label{control}
|| \Phi_t.x ||  \leq R_t . ||x|| \quad \forall x\in\R^d.
\end{align}

\begin{lem}
\label{lem: OST}
Let $U,V$ be two 
compact Borel subsets of $\R^d$. 
Define 
$$ T=T_{U,V}:=\inf\{t\geq 0 \ : \ \Phi_t .U \subset V \} . $$
Then,
\begin{align*}
\nu(V)\geq \Pr[T<\infty]\cdot \nu(U).
\end{align*}
\end{lem}

\begin{proof}
Let $M_t : =  \nu (\Phi_t^{-1} . V )$.
Because $\nu$ is $\mu$-stationary, this process is a (positive) martingale.
Indeed, 
\begin{align*}
\E [ M_{t+1} \ | \ \Phi_0, \ldots, \Phi_t ] & = \int d \mu(\psi) \nu( \psi^{-1} \Phi_t^{-1}  . V )
= \mu \ast \nu (\Phi_t^{-1}. V) = \nu( \Phi_t^{-1} .V) = M_t .
\end{align*}
For any $t>0$, by the Optional Stopping Theorem (see \eg \cite{Durrett}) at time $T \wedge t$,
we have that 
\begin{align*}
\nu(V) & = M_0 = \E [ M_{T \wedge t} ] \geq \E [ M_T \1{ T < t } ] \geq \nu(U) \cdot \Pr [ T < t ] ,
\end{align*}
where we have used that $U \subset \Phi_T^{-1} . V$ a.s.

Sending $t \to \infty$ completes the proof.
\end{proof}

\begin{lem}
\label{lem: reduction}
Fix a constant $k_0>1$. Define the following subsets of $\S_1$: 
\begin{align*}
&V_0=\{(a,b):\ k_0^{-1}\leq a \leq k_0,\ |b|\leq k_0 \ \}, \\
&V_z=V_0 \cdot (z^{-1},0)=\{ (a,b): \ k_0^{-1}\cdot z^{-1} \leq a \leq k_0\cdot z^{-1},\ |b|\leq k_0 \ \}.
\end{align*}
We have 
\begin{align*}
\nu( [-2k_0,2k_0]^d )\geq \Pr[T_{V_z}<\infty]\cdot \nu([-z,z]^d),
\end{align*}
where $T_{V_z}=\inf\{t:\ R_t\in V_z\}$.
\end{lem}

\begin{proof}
Note that if $R_t\in V_z$ and $x\in \R^d$ satisfies $||x||\leq z$, then $R_t . ||x||\leq 2k_0$,
and thus by \eqref{control}, we have $||\Phi_t(x)||\leq 2k_0$. 
So if we take $V=[-2k_0,2k_0]^d$ and $U=[-z,z]^d$, we get that $T_{V_z} \geq T_{U,V}$.
Applying Lemma \ref{lem: OST},
\begin{align*}
\nu(V)\geq \Pr[T_{U,V}<\infty]\cdot \nu(U)\geq \Pr[T_{V_z}<\infty]\cdot \nu(U) .
\end{align*}
\end{proof}

\begin{proof}[Proof of  Lemma \ref{lem: linear growth}]
Lemma 3.4(2) in \cite{BB15} states that under our assumptions on $\mu$, 
there exists some $k_0 > 1$ and $\delta>0$, 
such that for $V_0, V_z$ as in Lemma \ref{lem: reduction},
we have for all $z \geq 1$,
$$ \Pr[T_{V_z}<\infty]>\frac{\delta}{1+\log z}. $$
Plugging this into Lemma \ref{lem: reduction}, 
we arrive at
$$ \nu( [-z,z]^d) \leq \nu( [-2k_0, 2k_0]^d ) \cdot \frac{1+ \log z }{\delta} , $$
for all $z \geq 1$.
This proves \eqref{measure growth}, which is sufficient to obtain Lemma \ref{lem: linear growth},
as remarked above (before \eqref{measure growth}).
\end{proof}

\subsection{From locally compact groups to $\S_d$.}

\label{scn:from LC to Sd}

In this section we will overview the reduction from general CGLC connected groups to the case of closed subgroups of $\S_d$.

Let $G$ be a connected CGLC group, and $\mu$ a courteous measure on $G$. Let $(X_t)_{t\geq 0}$ be a $\mu$-random walk on $G$, \ie $X_0=1$ and the increments $X_t^{-1}X_{t+1}$ are independent $\mu$-random variables. Let $H\leq G$ be a finite-index subgroup, and define the return time to $H$ by $\tau_H=\inf\{t:\ X_t\in H\}$. It is well known that since $H$ is of finite index, $\tau_H$ is almost surely finite. Define the {\em hitting measure}
$\mu_H$ on $H$ by 
$$
\mu_H(A)=\Pr[X_{\tau_H}\in A].
$$
In \cite{BE95}, it is shown that $\mu_H$ is a courteous measure on $H$. It is then shown, that if $f_H$ is a $\mu_H$-harmonic function on $H$, then $f(g):=\E[f_H(X_{\tau_H})|\ X_0=g]$ is a $\mu$-harmonic function on $G$. In fact, we have:
\begin{prop}[\cite{BE95} Lemma 3.4, \cite{MY16} Proposition 3.4] 
\label{prop: courteous motivation}
Let $G$ be a CGLC group, $\mu$ a courteous measure, and $H$ a finite index subgroup. Then $\mu_{H}$ is a courteous measure on $H$. 
Moreover, for any $k$ 
the restriction map $f\mapsto f|_{H}$ is a linear isomorphism from $\HF_{k}(G,\mu)$ to $\HF_{k}(H,\mu_{H})$.
\end{prop}
Put simply, by passing to a finite index subgroup, the space of harmonic function of polynomial growth of degree at most $k$ is essentially the same. This proposition indicates that courteous measures provide a suitable framework to prove Conjecture \ref{conj:main}.

One can also pass to a continuous image of the group. Let $\pi:G\to Q$ be a continuous surjective 
homomorphism. In \cite[Lemma 3.1]{BE95}, it is shown that if $\mu$ is courteous on $G$ then $\mu_Q:=\mu\circ\pi^{-1}$ is 
courteous on $Q$. It is then straightforward to show that if $f_Q$ is a $\mu_Q$-harmonic function with linear growth on $Q$, then $f=f_Q\circ \pi$ is a $\mu$-harmonic function with linear growth on $G$. 

By the above, we may freely pass to finite-index subgroups and quotients, 
as long as the growth does not change.

\begin{proof}[Proof of Lemma \ref{lem: exists HF}]
A theorem of Yamabe and Gleason (see \eg \cite{mont_zipp}) tells us that 
since $G$ is connected, 
for any open neighborhood $U$ of the identity in $G$,
we may find a compact normal subgroup $K \subset U$, such that $G/K$ is a Lie group.
Since $K$ is compact, $G/K$ and $G$ have the same growth.
So we may, without loss of generality, assume that $G$ is a connected Lie group.

Using the Yamabe-Gleason theorem, Jenkins proves in \cite{jenkins} that $G$ must have 
either polynomial growth or exponential growth.  

If $G$ has polynomial growth, then we can just take $G' \leq \S_1$ to be isomorphic to $\R$, 
with $\mu'$ uniform on the isomorphic copy of $[-1,1]$.  In this case 
it is not difficult to prove that $\HF_1(G',\mu')$ is just the space of affine transformations,
so $\dim \HF_1(G',\mu') = 2$ and there are no non-constant positive harmonic functions.

So assume that $G$ has exponential growth.

We now proceed similarly to the proof of Theorem 1.4 in \cite{BE95}:

As a connected Lie group of exponential growth, 
Lemma 3.10 in \cite{BE95} tells us that we may find a homomorphism $\pi : G \to \mathsf{GL}(\R^d)$
such that $\pi(G)$ has exponential growth. As a homomorphic image of $G$,
the dimension of $\HF_1$ on $\pi(G)$ cannot increase, as mentioned above.
So we may further assume without loss of generality that $G$ is a connected closed subgroup 
of $\mathsf{GL}(\R^d)$, that has exponential growth.

If $G$ is non-amenable, it admits non-constant bounded harmonic functions, as is well known
(see \eg \cite[Theorem 1.1]{BE95}).
So assume that $G$ is amenable.
In this case, 
by Proposition 3.9 in \cite{BE95} (see also \cite[Definition 3.5]{BE95}) 
there exist a finite-index normal subgroup $H \lhd G$,
and a homomorphism $\rho : H \to \S_d$ such that $\rho(H)$ has exponential growth.
Because $H$ has finite index in $G$, we get that the dimension of 
$\HF_1(H,\mu_H)$ and $\HF_1(G,\mu)$ are
the same, completing the proof.
\end{proof}

\section{Infinite dimensional orbit}
\label{sec: infinite dim'l orbit}

In this section we prove
Theorem \ref{thm: convRW}.

\begin{lem}
\label{lem: convRW FI}
Let $G$ be a CGLC group and let $\mu$ be a courteous measure on $G$.
Let $H \leq G$ be a subgroup of finite index and let $\mu_H$ be the hitting measure.

If $f:G \to \R$ converges along random walks (with respect to $\mu$) then 
the restriction $f \big|_H$ converges along random walks (with respect to $\mu_H$).
\end{lem}

\begin{proof}
Let $(X_t)_t$ be a $\mu$-random walk started at $X_0=y \in H$.
Let $\tau_0 = 0$ 
and let $\tau_{n+1} = \inf \{ t \geq \tau_n +1 \ : \ X_t \in H\}$ be the successive return times to $H$.
So $(Y_n : = X_{\tau_n})_n$ is a $\mu_H$-random walk started at $Y_0 = y$.

Since $(f(X_t))_t$ converges a.s., also $(f(Y_n))_n$ converges a.s.\ as a sub-sequence.
This holds for arbitrary $y \in H$ completing the proof.
\end{proof}

\begin{lem}
\label{lem: compact}
Let $G$ be a compact group and $\mu$ a courteous measure on $G$.

If $f:G \to \R$ is a continuous function that converges along random walks then $f$ is constant.
\end{lem}

\begin{proof}
Assume for a contradiction that $f$ is non-constant.
Let $x \in G$ be such that $f(x) \neq f(1)$. $f$ is continuous, so we may choose 
two open neighborhoods $x \in U , 1 \in V$ 
such that 
$$ \sup_{z \in U} |f(x) - f(z)| < \tfrac12 | f(x) - f(1) | \qquad \textrm{and} \qquad 
\sup_{y \in V} |f(1)  - f(y) | < \tfrac12 | f(x) - f(1) | . $$
Specifically, 
$V \cap U = \emptyset$ and $U,V$ have positive Haar measure.
Also, for any $z \in U, y \in V$ we have $f(z) \neq f(y)$.

Now,
the ergodic theorem tells us that for any measurable subset $A \subset G$ we have that
$\tfrac1t \1{A}(X_t) \to \lambda(A)$ a.s.\ where $\lambda$ is the normalized Haar probability measure on $G$.
Thus, a.s.\ the sequence $(f(X_t)_t)$ contains an accumulation points in any open set of positive Haar measure,
contradicting convergence along random walks.
\end{proof}

\begin{lem}
\label{lem: co compact}
Let $G$ be an amenable CGLC group
and let $\mu$ be a courteous measure on $G$.
Let $f:G \to \R$ be a continuous function that converges along random walks.

Assume that there exists a co-compact normal subgroup $H \lhd G$ such that $H$
acts trivially on $f$.

Then, $f$ is constant.
\end{lem}

\begin{proof}
Since $H$ acts trivially on $f$, this induces a continuous function on the compact group $G/H$
via $\bar f(Hx) = f(x)$.
If we consider the projected random walk (\ie the process $(H X_t)_t$) 
on this compact group, then $\bar f$ converges along random walks
(because $f$ does).
Thus, by Lemma \ref{lem: compact} $\bar f$ is constant.
This implies that $f$ is constant as well.
\end{proof}

We require the notion of a type $\mathbf{S}$ action following \cite{BE95}. 

\begin{dfn}
Let $\rho:G \to GL(V)$ be an action of $G$ on a finite-dimensional real vector space $V$. We say that this action is of type $\mathbf{S}$ if there exists a compact subgroup $K$ of $GL(V)$, a continuous homomorphism $k:G \to K$ and a continuous homomorphism $a: G \to (0,\infty)$ such that 
$
\rho(g) = a(g) k(g)
$
for all $g \in G$. 
\end{dfn}

\begin{proof}[Proof of Theorem \ref{thm: convRW}]
We will denote $X_{t+1} = X_t U_{t+1}$ for $(U_t)_{t \geq 1}$ i.i.d.-$\mu$ random steps.

Let $V = \mathrm{span} (G.f)$ and assume that $\dim V = d < \infty$. 
Note that $\big( h(x X_t) \big)_t$ converges for all $h \in V$. 
Since functions in $V$ factor through the kernel of the $G$-action, we may assume that $G \leq GL(V)$. 

Under this assumption, $G$ is now an amenable linear group.  
A result by Guivarc'h \cite{Gui73} 
states that 
there exists a finite index normal subgroup $G'$ of $G$, 
for which there is a finite sequence $\{0\} = V_0 \subset V_1 \subset ... V_n = V \cong \R^d$ of 
$G'$-invariant linear subspaces of $V$ 
such that the action of $G'$ on each $V_{i+1} / V_i$ is of type $\mathbf{S}$. By 
Lemma \ref{lem: convRW FI}, we may, without loss of generality, pass to the finite index subgroup,
since we are only required to prove that $G'$ acts trivially on $f$, due to Lemma \ref{lem: co compact}.
So we assume that a sequence $\{0\} = V_0 \subset V_1 \subset ... V_n = V$
exists with respect to $G$. Specifically, by an appropriate choice of basis $B$ we have
\begin{align}
\label{eq: type S matrix form}
[x]_B =  
\begin{pmatrix}
a_1(x) k_1(x) & x_{12} 			&  \cdots   &	x_{1n}	\\
0			  & a_2(x)k_2(x)	&  \cdots	&	x_{2n}	\\  
\vdots		  & 				&  \ddots	&	\vdots  \\
0			  &	\cdots			&			&	a_n(x) k_n(x)
\end{pmatrix}
\end{align}
where 
$k_i:G \to K_i\subset \mathrm{GL}(\R^d)$ and $a_i :G \to (0,\infty)$ are homomorphisms
and $K_i$ is a compact subgroup of $\mathrm{GL}(\R^d)$. 

Let $H \lhd G$ be the kernel of the homomorphism 
$x \mapsto (k_1(x) , \ldots, k_n(x))$.
So $G/H$ is isomorphic to the compact group $k_1(G) \times \cdots \times k_n(G)$.
Also, for any $x \in H$ we have that $k_j(x) = I$.

{\bf Step I.}
First we show that $V_1$ is the space of constant functions
(so $G$ acts trivially on $V_1$).

For any $x \in G$ and $h \in V_1$ we have 
$$
[x.h]_B = [x]_B [h]_B = a_1(x)k_1(x)  [h]_B .
$$
This is a slight abuse of notation, since we regard 
$a_1(x) k_1(x)$ as acting on the whole space $\R^d$, by 
identifying
\begin{align}
\label{eqn:matrix}
 a_j(x) k_j(x) & = 
\begin{pmatrix}
I & 0 & 0 &  \cdots   & 0 \\
0 & \ddots	&  \cdots	& & 0  \\  
\vdots  & 	&  a_j(x) k_j(x)	& & \vdots  \\
\vdots & & & \ddots & \vdots \\
0  &	\cdots &	& 0 & I
\end{pmatrix}
\end{align}

Consider $\IP{ \cdot, \cdot}$, the inner product on $\R^d$. 
For any $y \in G$ the map $h \mapsto h(y)$ is a linear functional on $V$,
so by the Riesz  representation theorem there exists $v_y \in \R^d$ such that
$\IP{ [h]_B , v_y } = h(y)$ for all $h \in V$.

Let $h \in V_1$ be any function.
Then for any $x \in H$ and $y \in G$, since $k_1(x) = I$,
$$ h(x) = x^{-1} . h(y) = a_1(x)^{-1} \cdot \IP{ [h]_B , v_y } = a_1(x)^{-1} \cdot h(y) . $$
Thus, for any $x \in H$ we have $h(x^{-n}) = a_1(x)^{n} h(1)$.
Because we assumed that $h$ grows sub-exponentially, this implies that $a_1(x) = 1$ for all $x \in H$.
So $H$ acts trivially on any $h \in V_1$.
By Lemma \ref{lem: co compact} this implies that $V_1$ is the space of constant functions.

{\bf Step II.}
We now show that $H$ acts trivially on $V_2$. 
(If $d=1$ this step is redundant, since we have already shown that $V=V_1$ is the space of constant functions).

Let $h \in V_2$.
Let $\delta_1 \in \R^d$ be the vector with $1$ in the first coordinate and $0$ elsewhere.
Note that since $V_1$ is the space of constant functions,
$\delta_1 \perp v_y - v_1$ for all $y \in G$.

For all $x \in G$ we have
$$ [x.h]_B = a_2(x) \cdot k_2(x) [h]_B + \IP{ (0, x_{12} , \ldots, x_{1n}) , [h]_B } \cdot \delta_1 . $$
The important observation here is that the coefficient of $\delta_1$ above depends only on $x$ and 
not on the specific point of evaluation of the function $\delta_1$.
So if $x \in H$ then for any $y \in G$,
$$ h(xy) - h(x) = a_2(x)^{-1}  \cdot \IP{ [h]_B , v_y - v_1 } = a_2(x)^{-1} ( h(y) - h(1) ) . $$
This implies that for any $x \in H$,
\begin{align*}
h(x^{-n}) - h(1) & 
= \sum_{j=0}^{n-1} a_2(x)^j \cdot (h(x^{-1}) - h(1) ) 
= \tfrac{ a_2(x)^n - 1}{a_2(x) - 1} \cdot ( h(x^{-1} ) - h(1) ) .
\end{align*}
As before, since we assumed that $h$ has sub-exponential growth, this implies that 
$a_2(x) = 1$ for any $x \in H$, which is to say that $H$ acts trivially on any $h \in V_2$.

Thus, by Lemma \ref{lem: co compact} any $h \in V_2$ is constant, 
implying that $G$ acts trivially on $V_2$.

{\bf Conclusion.}
Since $V_2$ is the space of constant functions, it must be that actually $d=1$
and $V_1=V$ is the space of constant functions, and that originally in \eqref{eq: type S matrix form}
the matrices were all the identity matrix.
This shows that $G$ acts trivially on the orbit of $f$ and specifically on $f$.
\end{proof}

Following the statement of Theorem \ref{thm: convRW} we remarked that
in the case where $G$ is finitely generated this theorem holds without the 
sub-exponential growth assumption.
Since the proof is almost identical we only sketch the proof of this observation.

\begin{proof}[Sketch of proof]
As in the proof of Theorem \ref{thm: convRW} we arrive at 
a representation as in \eqref{eq: type S matrix form}.
Setting $H$ to be the co-compact subgroup which is the kernel of the map $x \mapsto (k_1(x) , \ldots, k_n(x))$,
we find that $H$ is of finite index (because the compact group $k_1(G) \times \cdots \times k_n(G)$
is actually finitely generated in this case, and thus finite).

Thus, by Lemma \ref{lem: convRW FI}, we may pass to
the subgroup $H$ instead of $G$.  
Then, the same reasoning as in Step I of the proof above gives that 
for any $h \in V_1$,
we have $h(X_t) = a_1(X_t)^{-1} \cdot h(1)$.
Thus, $\log \frac{ h(X_t) }{h(1) }$ is a symmetric random walk on the additive group $\R$.
Such a random walk can only converge if it is degenerate (see \cite{Durrett});
that is, if $h(X_t) = h(1)$ a.s.\ for all $t$.
Because $\mu$ is adapted this implies that $h$ is constant.

Once establishing that $V_1$ is the constant functions, as in 
Step II of the proof of Theorem \ref{thm: convRW}, we arrive at
\begin{align}
\label{eqn: h V2}
h(X_t y) - h(X_t) & = a_2(X_t)^{-1} \cdot (h(y) - h(1)) 
\end{align}
for any $y$ and any $t$, and for any $h \in V_2$.
Now, for a fixed $y$ there exist $n \in \N$ and $\alpha>0$ such that $\mu^n(y) > \alpha$.
Thus, as $t \to \infty$ for any $\eps>0$,
\begin{align*}
\alpha \cdot \Pr [ | h(X_t y) - h(X_t) | > \eps ] & \leq \Pr [ | h(X_{t+n}) - h(X_t) | > \eps ] \to 0 .
\end{align*}
That is, the left hand side of \eqref{eqn: h V2} converges to $0$ in probability.
However, as before, $( \log a_2(X_t) )_t$ is a symmetric random walk on $\R$,
implying that it can only converge if it is degenerate.
So it must be that $a_2 \equiv 1$ and we arrive at $h(xy) - h(x) = h(y)-h(1)$ for all $x,y \in H$.
This implies that $h - h(1)$ is a homomorphism from $H$ into the additive group $\R$.
Specifically, $(h(X_t) - h(1) )_t$ forms a symmetric random walk on $\R$,
and because this random walk must converge a.s., we obtain as before that $h$ is constant.
\end{proof}

\bibliographystyle{alpha}

\end{document}